\documentclass[12pt,reqno]{article}

\usepackage[usenames]{color}
\usepackage{amssymb}
\usepackage{graphicx}
\usepackage{amscd}

\usepackage[colorlinks=true,
linkcolor=webgreen,
filecolor=webbrown,
citecolor=webgreen]{hyperref}

\definecolor{webgreen}{rgb}{0,.5,0}
\definecolor{webbrown}{rgb}{.6,0,0}

\usepackage{color}
\usepackage{fullpage}
\usepackage{float}

\usepackage{graphics,amsmath,amssymb}
\usepackage{amsthm}
\usepackage{amsfonts}
\usepackage{latexsym}
\usepackage{epsf}

\usepackage{verbatim}
\usepackage{enumerate}

\usepackage{tikz}
\usepackage{caption}
\usepackage{capt-of}
\usepackage{forest}
\usepackage{subcaption}

\usepackage{mathtools}

\usepackage{array}

\usepackage[blocks]{authblk}

\DeclareMathOperator{\out}{out}

\makeatletter
\newcommand{\thickhline}{
    \noalign {\ifnum 0=`}\fi \hrule height 1pt
    \futurelet \reserved@a \@xhline
}
\newcolumntype{"}{@{\hskip\tabcolsep\vrule width 1pt\hskip\tabcolsep}}
\makeatother

\theoremstyle{plain}
\newtheorem{theorem}{Theorem}
\newtheorem{corollary}[theorem]{Corollary}
\newtheorem{lemma}[theorem]{Lemma}
\newtheorem{proposition}[theorem]{Proposition}

\newtheorem{innerfact}{Fact}
\newenvironment{fact}[1]
  {\renewcommand\theinnerfact{#1}\innerfact}
  {\endinnerfact}

\theoremstyle{definition}

\theoremstyle{remark}

\begin{document}

\title{Generalizing the Wythoff Array and other Fibonacci Facts to Tribonacci Numbers}
\author{Eric Chen}
\author{Adam Ge}
\author{Andrew Kalashnikov}
\author{Ella Kim}
\author{Evin Liang}
\author{Mira Lubashev}
\author{Matthew Qian}
\author{Rohith Raghavan}
\author{Benjamin Taycher}
\author{Samuel Wang}
\affil{PRIMES STEP}
\author{Tanya Khovanova}
\affil{MIT}

\maketitle

\begin{abstract}
In this paper, we generalize a lot of facts from John Conway and Alex Ryba's paper, \textit{The extra Fibonacci series and the Empire State Building}, where we replace the Fibonacci sequence with the Tribonacci sequence. We study the Tribonacci array, which we also call \textit{the Trithoff array} to emphasize the connection to the Wythoff array. We describe 13 new sequences.
\end{abstract}

\section{Introduction}

We stumbled upon a recent paper by John Conway and Alex Ryba \cite{ConwayRyba2016}, \textit{The extra Fibonacci series and the Empire State Building}. The title intrigued us, and our initial goal was to find the Empire State Building related to the Tribonacci sequences.

The Conway-Ryba paper \cite{ConwayRyba2016} studies the Wythoff array. The array itself initially appeared in relation to the Wythoff game introduced by Wythoff \cite{Wythoff1907}. It is a two-player impartial combinatorial game with deep connections to the Fibonacci numbers.

The Wythoff array is an infinite table that contains all integers once. The integers increase along each row and column. The integers in the array can be divided into pairs naturally, and each pair is a P-position in the Wythoff game. In this paper, we do not try to generalize the Wythoff game. We use an independent definition of the Wythoff array using Fibonacci numbers.

Every integer has a unique Zeckendorf representation, which can be considered as a Fibonacci-base representation. Namely, every integer can be represented as a sum of distinct Fibonacci numbers so that no two numbers are consecutive. This representation can be encoded as a string of ones and zeros without two consecutive ones. The $n$th column of the Wythoff array are numbers with $n-1$ zeros at the end of their Zeckendorf representation. For example, the first column consists of the numbers with their Zeckendorf representation ending in 1. A number $m$ not in the first column is a Fibonacci successor of the number $n$, which is located to the left of $m$ in the array. Conway and Ryba \cite{ConwayRyba2016} use the notation $m = \out(n)$.

Conway and Ryba \cite{ConwayRyba2016} extended the array to the left and found some natural border lines that symmetrically surround the central column of the array. The lines form a shape resembling the Empire State building.

The Conway-Ryba paper \cite{ConwayRyba2016} contains a lot of statements that are called facts. We generalized half of the statements to the Tribonacci numbers. We also went on different tangents and added more statements unrelated to the paper. The Wythoff array's analog is known and called the Tribonacci array \cite{DucheneRigo2008,Keller1972}. In this paper, we call it the Trithoff array to emphasize the connection to the Wythoff array. 

We did not actually find an analog of the Empire State building in the Trithoff array. This is because when extending the Trithoff array to the left, we lose the symmetry properties of the array. But we have found many other things.

Here we describe what is done in the paper together with a road map.

Section~\ref{sec:Preliminaries} covers the list of facts from Conway-Ryba paper \cite{ConwayRyba2016} that we generalize. We state well-known facts about Tribonacci numbers $T_n$, the Fibonacci word, and the Tribonacci word. We introduce the Tribonacci successor of integer $n$, which we denote as $\out(n)$, by analog with Conway and Ryba \cite{ConwayRyba2016}.

It is known that the Tribonacci sequence grows approximately as a geometric series $\alpha^n$, where $\alpha$ is the Tribonacci constant. In Section~\ref{sec:bounds} we estimate the value of $\out(n)$ as $\alpha n-0.85 < \out(n)<\alpha n+ 0.85$. We also study the difference $T_{n+1} - \alpha T_n$. We show that the sequence of such differences cannot have the same sign for any three consecutive terms and describe positive and negative records for such differences.

In Section~\ref{sec:trithoffarray}, we discuss the Tribonacci array, which we call the Trithoff array, to emphasize the connections to the Wythoff array. In Section~\ref{subsec:rowdiffsequences}, we discuss difference sequences for every row. We discuss how for a given row to find the row that represents the difference sequence. We prove that given a row $r$, we can find another row such that $r$ is its difference sequence if and only if row $r$ is not the row that repeats periodically even, even, odd,  and odd numbers. In Section ~\ref{subsec:columndiffsequences}, we describe the difference sequence of the first column and show that it consists of twos and threes. We also show that a number in a Trithoff array in column $c$ and row $r$ can be approximated as $\frac{r\alpha^c}{\alpha-1}$. 

In Section~\ref{sec:precolumns}, we extend the Trithoff array to the left. We study the columns $-2$, $-1$, and 0, which we call, by analog to Conway-Ryba \cite{ConwayRyba2016}, the pre-seed, seed, and wall. We describe these columns in terms of the first column. For example, we prove that the wall term $w$ is followed by $\out(w) - 1$ in the first column. The pre-seed, seed, and wall form new sequences, which we describe in full detail.

In Section~\ref{sec:multiples}, we prove that any positive Tribonacci-like sequence has its tail appearing in the Trithoff array. We also study how multiples of Tribonacci-like sequences appear in the array. We prove that such multiples appear in order. In addition, we show that the $n$th multiple of a Tribonacci-like sequence has the row number equal to 1 modulo $n$. We explain that when extending a positive Tribonacci-like sequence to the left, before some index $i$, the sequence cannot have three numbers with the same sign. Moreover, if we take the absolute values of the numbers before the index $i$ and reverse the sequence, it cannot be a Tribonacci-like sequence. This is a stark difference from the Fibonacci case.

In Section~\ref{sec:tribinarrynumbers}, we describe Fibbinary, Tribbinary, Fibternary, Tribternary numbers, and their properties.

We mention 23 existing sequences from the OEIS. We study and describe 13 new sequences: four of them are particular columns of the Trithoff array, three sequences are related to Tribonacci numbers, and six sequences are related to the Trithoff array but are not rows or columns.

\section{Preliminaries}\label{sec:Preliminaries}

\subsection{The extra Fibonacci series and the Empire state building}\label{sec:EmpireState}

We summarize results relevant to us from the paper ``The extra Fibonacci series and the Empire State Building'' \cite{ConwayRyba2016} by John Conway and Alex Ryba. The paper has a lot of statements in the form of facts. We present here some facts from the paper that we plan to generalize.

In the Fibonacci sequence 0, 1, 1, 2, 3, 5, 8, $\ldots$ (A000045), each term is the sum of the previous two. We say that this sequence follows the \textit{Fibonacci rule}. Integer sequences that follow the Fibonacci rule and end in positive integers are called \textit{extraFib series} or \textit{extraFibs}, see \cite{ConwayRyba2016}. They use the word series to emphasize that the sequences can be extended in both directions.

The \textit{Zeckendorf representation} of an integer $k$ is its expression as a sum of positive Fibonacci numbers, where each Fibonacci number can only be used once, and there can be no two consecutive Fibonacci numbers in the sum.

\begin{fact}{2}
The Zeckendorf expansion of $n$ is unique.
\end{fact}

Suppose integer $n$ has Zeckendorf representation
\[n = F_{i_1} + \cdots +F_{i_k}.\] 
We can denote \textit{the Fibonacci successor} of $n$ as $\out(n)$, where it is defined as:
\[\out(n) = F_{i_1+1} + \cdots +F_{i_k+1}.\]

\begin{fact}{1}
The function $\out(n)$ is well-defined.
\end{fact}

For the next fact, we denote the golden ratio as $\phi$.

\begin{fact}{3}
The unique integer in the open unit interval $(\phi n - \phi^{-2}, \phi n + 1 - \phi^{-2})$ is $\out(n)$.
\end{fact}

The Wythoff array is an infinite table $W_{m,n}$, with $m,n > 0$, where $W_{m,n}$ denotes the entry in row $m$ and column $n$ of the array, and where 
\begin{itemize}
\item $W_{{m,1}}=\left\lfloor \lfloor m\varphi \rfloor \varphi \right\rfloor$,
\item $W_{m,2}=\left\lfloor \lfloor m\varphi \rfloor \varphi ^{2}\right\rfloor$,
\item $W_{m,n} = W_{m,n-2} + W_{m,n-1}$ for $n > 2$.
\end{itemize}

We see that each row of the array is an extraFib, and thus, we can continue each sequence to the left. Table~\ref{table:theGardenCR} shows the corner of the Wythoff array with two more columns on the left. 
\begin{table}[ht!]
\begin{center}
    \begin{tabular}{c|c!{\vrule width 1pt}ccccccccc}
         0&1&1&2&3&5&8&13&21&34& \dots\\
         1&3&4&7&11&18&29&47&76&123& \dots\\
         2&4&6&10&16&26&42&68&110&178& \dots\\
         3&6&9&15&24&39&63&102&165&267& \dots\\
         4&8&12&20&32&52&84&136&220&356& \dots\\
         5&9&14&23&37&60&97&157&254&411& \dots\\
         6&11&17&28&45&73&118&191&309&500& \dots\\
         7&12&19&31&50&81&131&212&343&555& \dots\\
         8&14&22&36&58&94&152&246&398&644& \dots\\
         9&16&25&41&66&107&173&280&453&733& \dots\\
         10&17&27&44&71&115&186&301&487&788& \dots\\
    \end{tabular}
\end{center}
\caption{The Garden State}
\label{table:theGardenCR}
\end{table}

Conway and Ryba \cite{ConwayRyba2016} call the table the Garden State. The numbers in the first column are called the \textit{seed} terms; in the second column, the \textit{wall} terms; and other numbers are called \textit{the Garden}. The Garden grows out in the sense that for any term $n$ in the Garden, the next term to the right is $\out(n)$.

The first row is Fibonacci numbers, while the second row is Lucas numbers. The next three rows are 2, 3, and 4 times the Fibonacci numbers. The sixth row is called the Pibonacci numbers (A104449). When extended to the left, the first few terms of the sequence look like the digits of the number Pi: 3, 1, 4, 5, 9.

We can see that the first column of the Wythoff array consists of integers whose Zeckendorf representation ends in 1. It is easy to prove the following structure of Table~\ref{table:theGardenCR}.

\begin{fact}{5}
(a) A garden term $n$ is followed by $\out(n)$.
(b) A seed term $s$ is followed by $\out(s) + 1$.
(c) A wall term $w$ is followed by $\out(w) - 1$.
\end{fact}

We also have the following amazing fact about how natural numbers appear in the Garden State.
\begin{fact}{6}
Every positive integer appears exactly once in the Garden and once as a seed, and zero also appears just once as a seed.
\end{fact}

The next fact explains how extraFibs appear in the array.

\begin{fact}{7}
Every series that satisfies Fibonacci's rule and ends with positive integers is represented in the Garden State.
\end{fact}

By the words ``is represented'' they mean that given an extraFib (which may be extended in both directions), we can find a number $k$, such that all the terms of the extraFib starting from index $k$ form a row in the Garden State.

\begin{fact}{8}
If $X_n$ is any extraFib series, so too is any positive multiple $mX_n$.
\end{fact}

\begin{fact}{9}
The multiples of any extraFib series appear in order in the Garden State.
\end{fact}

The parts above related to the Garden, also known as the Wythoff array, are widely known. 

The paper \cite{ConwayRyba2016} also describes how to find each particular extraFib in the Garden State. Given the extraFib $S_n$, we expand it to the left and right and find the out values of every term. The wall term $S_w$ corresponds to the largest $w$, such that $S_{w+1} = \out(w) - 1$.

It is well-known that if we continue extraFibs to the left, negative integers will alternate with positive integers. Thus, Conway and Ryba defined a \textit{reversal} of an extraFib, which is obtained by reversing the order and changing negative signs to positive ones. One can check the following fact.

\begin{fact}{10}
The reversal of an extraFib series is also an extraFib series.
\end{fact}

In our paper, we prove analogs of facts described above for the Tribonacci case. We denote our analogs of these facts with a letter T after the number.

But first, we need some more information about Fibonacci numbers.

\subsection{The Fibonacci word}

A \textit{Fibonacci word} is an infinite word in a two-letter alphabet formed by the following infinite process. 

Let $S_{0} = a$ and $S_{1} = ab$. Now $S_{n}=S_{n-1}S_{n-2}$, the concatenation of the previous word and the one before that.

The infinite Fibonacci word is the limit $S_{\infty}$, that is, the unique infinite sequence that contains each $S_{n}$, for finite $n$, as a prefix.

The first few terms of the Fibonacci word are described in sequence A003849:
\[abaababaabaab\ldots.\]

The positions of $a$'s form the lower Wythoff sequence A000201. These are the numbers from the odd-numbered columns in the Wythoff array. Moreover, these are the wall numbers. The positions of $b$'s form the complementary sequence: the upper Wythoff sequence A001950.

Now we describe the background for the Tribonacci numbers.

\subsection{Tribonacci numbers}\label{sec:tribonacci}

The Tribonacci numbers $T_n$ start with $T_0 = 0$, $T_1 = 0$, $T_2 = 1$, and each Tribonacci number thereafter is calculated by summing up the previous 3 numbers:
\[T_{n} = T_{n-1} + T_{n-2} + T_{n-3}\]
for $n > 2$. The Tribonacci sequence, A000073, would thus be 
\[0,\ 0,\ 1,\ 1,\ 2,\ 4,\ 7,\ 13,\ 24,\ 44,\ 81,\ 149,\ 274,\ 504,\ 927,\ 1705,\ \ldots\]. 

The \textit{Tribonacci representation} of an integer $k$ is its expression as a sum of positive Tribonacci numbers, where each Tribonacci number can only be used once, and there can be no 3 consecutive Tribonacci numbers in the sum. Note that number 1 can only be used once, even though it appears twice in the sequence.

The Tribonacci representation can be viewed as numbers written in the Tribonacci base. We denote the Tribonacci representation of integer $N$ as $(N)_T$ and the evaluation of a string $w$ written in the Tribonacci base as $[w]_T$. For example, 9 can be expressed as $7+2$, making 1010 a Tribonacci representation of 9, where the rightmost digit 0 represents that there is no 1 in the sum, the 1 to the left of that represents there is a 2, the 0 to the left of that represents that there are no 4s, and finally the leftmost 1 represents that there is one 7. Thus, we write $(9)_T = 1010$ and $[1010]_T = 9$. %Some more examples: $(18)_T = 10101$ as $18 = 13+4+1$, $22_T = 11010$ as $22 = 2+7+13$, and $(25)_T = 100001$ as $25 =24+1$.

One can find a unique Tribonacci representation by using a greedy algorithm. Start by finding the largest Tribonacci number that is less than our number and subtract it. Then, repeat the process. This method ensures that no three consecutive Tribonacci numbers are used. This is similar to Fact 2 in the Conway and Ryba paper \cite{ConwayRyba2016} in that the Zeckendorf representation is unique. The following theorem is proved in \cite{Keller1972}.

\begin{fact}{2T}
Every natural number has a unique Tribonacci representation.
\end{fact}

Consider the characteristic equation for the Tribonacci sequence:
\[x^3-x^2-x-1=0.\]

We can define the Tribonacci successor similar to how the Fibonacci successor is defined. Suppose integer $n$ has a Tribonacci representation
\[n = T_{i_1} + \cdots +T_{i_k}.\] 
We define the \textit{Tribonacci successor} of $n$ as $\out(n)$, where
\[\out(n) = T_{i_1+1} + \cdots +T_{i_k+1}.\] 

The following analog of Fact 1 follows from the uniqueness of the Tribonacci representation.

\begin{fact}{1T}
The Tribonacci successor is well-defined.
\end{fact}

The Tribonacci numbers grow approximately as a geometric series with the ratio equal to the Tribonacci constant. Thus,
\[\out(n) \approx \alpha n.\]

\subsection{The Tribonacci word}\label{sec:tribonacciword}

The \textit{Tribonacci word} is the limit of the sequence of words $W(n)$, where $W(n)$ is a string of digits $a$, $b$, and $c$ formed in the following manner: $W(0) = a$, $W(1) = ab$, and $W(2) = abac$. Then $W(n) = W(n-1)W(n-2)W(n-3)$ is the concatenation of the previous Tribonacci word, the one before it, and the one before that:
\[abacabaabacababac.\]

The following theorem \cite{DucheneRigo2008} by Duch\^{en}e and Rigo connects the positions of different letters in the Tribonacci word with the Tribonacci representations of said positions.

\begin{theorem}[\cite{DucheneRigo2008}]
\label{thm:abcCorrespondsTo0-01-11}
The $n$th symbol of the Tribonacci word is $a$, $b$, or $c$ if the Tribonacci representation of $n-1$ ends in 0, 01, or 11, respectively.
\end{theorem}

This is equivalent to the following statement. The $n$th symbol of the Tribonacci word is $a$, $b$, or $c$ if the Tribonacci representation of $n$ ends in the number of trailing zeros that equal 0, 1, or 2 modulo 3, respectively. Equivalently, The $n$th symbol of the Tribonacci word is $a$, $b$, or $c$, if $n$ is in the column number of the Trithoff array, that has remainder 1, 2, or 0 modulo 3, respectively.

The following three complementary sequences correspondingly describe the positions of letters $a$, $b$, and $c$ in the Tribonacci word.

Sequence A003144 describes the positions of the letter $a$ in the Tribonacci word.
\[1,\ 3,\ 5,\ 7,\ 8,\ 10,\ 12,\ 14,\ 16,\ 18,\ 20,\ 21,\ 23,\ 25,\ 27,\ 29,\ 31,\ \ldots.\]
It is known \cite{DSS2019} that A003144($n$) is always either $\lfloor \alpha n\rfloor$ or $\lfloor \alpha n + 1 \rfloor$ for all $n$, where $\alpha$ is the Tribonacci constant. In other words, $\alpha n - 1 < A003144(n) < \alpha n + 1$.

Sequence A003145 describes the positions of the letter $b$ in the Tribonacci word.
\[2,\ 6,\ 9,\ 13,\ 15,\ 19,\ 22,\ 26,\ 30,\ 33,\ 37,\ 39,\ 43,\ 46,\ 50,\ 53,\ 57,\ \ldots.\]

Sequence A003146 describes the positions of the letter $c$ in the Tribonacci word.
\[4,\ 11,\ 17,\ 24,\ 28,\ 35,\ 41,\ 48,\ 55,\ 61,\ 68,\ 72,\ 79,\ 85,\ 92,\ 98,\ 105, \ldots.\]

\section{Bounds on the successor}\label{sec:bounds}

The Tribonacci sequence grows approximately as a geometric progression with the ratio $\alpha$. It is well-known that for $n > 2$, we have $T_n < \alpha^{n-2} < T_{n+1}$, see \cite{BravoLuca2013}.

We are interested in a Tribonacci analog of Fact 3, which provides bounds on the Fibonacci successor. We start by estimating the growth of the Tribonacci sequence in terms of the roots of the characteristic equation defined in the previous section.

\begin{lemma}
\label{lemma:alphadiff}
We have
\[|T_{n+1}-\alpha T_n | <\frac{2\alpha^{-\frac{n}{2}}}{|\beta-\gamma|}.\]
\end{lemma}

\begin{proof}
The explicit formula for $T_n$ in terms of roots of the characteristic equation is well-known \cite{Spickerman1980}:
\[T_n=\frac{\alpha^n}{(\alpha-\beta)(\alpha-\gamma)}+\frac{\beta^n}{(\beta-\alpha)(\beta-\gamma)}+\frac{\gamma^n}{(\gamma-\alpha)(\gamma-\beta)}.\]
We compute 
\begin{multline*}
T_{n+1}-\alpha T_n=\frac{\alpha^{n+1}}{(\alpha-\beta)(\alpha-\gamma)}+\frac{\beta^{n+1}}{(\beta-\alpha)(\beta-\gamma)}+\frac{\gamma^{n+1}}{(\gamma-\alpha)(\gamma-\beta)} - \\
 \frac{\alpha\alpha^n}{(\alpha-\beta)(\alpha-\gamma)}+\frac{\alpha\beta^n}{(\beta-\alpha)(\beta-\gamma)}+\frac{\alpha\gamma^n}{(\gamma-\alpha)(\gamma-\beta)} = \frac{\beta^n}{\beta-\gamma}+\frac{\gamma^n}{\gamma-\beta}.
\end{multline*}
Taking the absolute value of both sides, we get $|T_{n+1}-\alpha T_n| \leq |\frac{\beta ^n}{\beta - \gamma}| + |\frac{\gamma ^n}{\gamma - \beta}|$. Since $|\beta|=|\gamma|$, we get  $|T_{n+1}-\alpha T_n| \leq 2\frac{|\beta| ^n}{|\beta - \gamma|}$. Since $|\beta|=|\gamma|$ and $\alpha\beta\gamma=1$, we know that $|\beta|^n=\alpha^{-\frac{n}{2}}$. This lemma follows.
\end{proof}

\begin{fact}{3T}
If $\out(n)$ is the Tribonacci successor of $n$, then $\alpha n-0.85 < \out(n)<\alpha n+ 0.85$.
\end{fact}

\begin{proof}
Suppose the Tribonacci representation of $n$ is $T_a+T_b+\cdots$, where $a \ge 3$. Then the successor is $T_{a+1}+T_{b+1}+\cdots$. By Lemma~\ref{lemma:alphadiff}, we have
\[| \out(n)-\alpha n|<\frac{2\alpha^{-\frac{a}{2}}}{|\beta-\gamma|}+\frac{2\alpha^{-\frac{b}{2}}}{|\beta-\gamma|}+\cdots.\]
Given that the Tribonacci representation excludes three consecutive Tribonacci numbers, we can bound the expression as 
\begin{multline*}
\frac{2}{|\beta-\gamma|}\left(\alpha^{-\frac{3}{2}} + \alpha^{-\frac{4}{2}} + \alpha^{-\frac{6}{2}} + \alpha^{-\frac{7}{2}} + \cdots\right) = \\ 
\frac{2}{|\beta-\gamma|}\left(\alpha^{-\frac{3}{2}} + \alpha^{-\frac{4}{2}}\right)\sum_{m=0}^\infty \alpha^{-\frac{3m}{2}} = \frac{2(\alpha^{-\frac{3}{2}} + \alpha^{-\frac{4}{2}})}{|\beta-\gamma|(1-\alpha^{-\frac{3}{2}} )} \approx 1.91746 < 2.
\end{multline*}

We can improve this bound further by adding computational results. We calculated the largest possible difference for numbers that have a Tribonacci representation up to 35 Tribonacci digits. The largest possible difference was less than 0.849. By a similar calculation to the one above, the larger digits can contribute to the difference no more than 
\[\frac{2(\alpha^{-\frac{36}{2}} + \alpha^{-\frac{37}{2}})}{|\beta-\gamma|(1-\alpha^{-\frac{3}{2}} )} < 0.0001.\]
Thus, the maximum difference is not more than 0.85.
\end{proof}

Consider the sequence of integers $n$ such that $T_{n+1} - \alpha T_n$ is positive (This is now sequence A352719.): 
\[0,\ 1,\ 3,\ 4,\ 6,\ 7,\ 9,\ 10,\ 12,\ 15,\ 18,\ 21,\ 24,\ 26,\ 27,\ 29,\ 30,\ 32,\ 33,\ 35,\ 36,\ 38,\ 41,\ 44,\ \ldots.\]
%0, 1, 3, 4, 6, 7, 9, 10, 12, 15, 18, 21, 24, 26, 27, 29, 30, 32, 33, 35, 36, 38, 41, 44, 47, 50, 52, 53, 55, 56, 58, 59, 61, 62, 64, 67, 70, 73, 76, 78, 79, 81, 82, 84, 85, 87, 88, 90, 93, 96, 99,

Let us denote this sequence as $P(k)$, where $P(0) = 0$ and $P(1) = 1$. We denote by $Q(n)$ the complementary sequence, with $Q(1) = 2$. The sequence $Q(n)$ is then the sequence of integers $n$ such that $T_{n+1} - \alpha T_n$ is negative (This is now sequence A352748.):
\[2,\ 5,\ 8,\ 11,\ 13,\ 14,\ 16,\ 17,\ 19,\ 20,\ 22,\ 23,\ 25,\ 28,\ 31,\ 34,\ 37,\ 39,\ 40,\ 42,\ 43,\ 45,\ \ldots.\]
%2, 5, 8, 11, 13, 14, 16, 17, 19, 20, 22, 23, 25, 28, 31, 34, 37, 39, 40, 42, 43, 45, 46, 48, 49, 51, 54, 57, 60, 63, 65, 66, 68, 69, 71, 72, 74, 75, 77, 80, 83, 86, 89, 91, 92, 94, 95, 97, 98,

\begin{proposition}
The value $T_{n+1} - \alpha T_n$ cannot have the same sign for any three consecutive integers $n$.
\end{proposition}

\begin{proof}
We know that $T_{i+1}-\alpha T_{i}=2\Re(\frac{\beta^i}{\beta-\gamma})$, where $\Re$ is the real part of a complex number. Now $2\Re(\frac{\beta^a}{\beta-\gamma})=\frac{2| \beta^a|}{\beta-\gamma}\sin(a\psi)$, where $\psi$ is the polar angle coordinate of $\beta$. We can calculate $\psi=\pm 2.17623$ radians, so $| 2\psi|>\pi$. This means there are not three consecutive integers where $T_{i+1}-\alpha T_i$ has the same sign.
\end{proof}

It follows that $P(n) - P(n-1) < 3$, and the sequence $P(n)$ does not contain three consecutive numbers. An analogous statement is true for $Q(n)$.

We now want to find numbers that create new records in the differences.

\begin{proposition}
Numbers that create new records in positive differences have the following properties.
\begin{enumerate}
\item If $k$ is a positive record, then all the indices in its Tribonacci representation belong to the sequence $P(n)$.
\item For any $k$, the number $T_{P(2)} + T_{P(3)} + \cdots + T_{P(k-1)} + T_{P(k)}$, creates a new record for positive differences.
\item For every $t > 1$, there exists $N_0$ such that the $N$th positive (negative) record for any $N > N_0$ contains $P(t)$.
\end{enumerate}
Similar statements are true for negative differences.
\end{proposition}

\begin{proof}
Consider the Tribonacci representation of $n$. Suppose $n$ achieves a new positive record, and suppose its Tribonacci representation contains $T_j$ such that $\out(T_j) - \alpha T_j < 0$; that is, $j\neq P(t)$ for any $t > 1$. Then consider the number $n - T_j$. The corresponding difference $\out(n-T_j) - \alpha(n - T_j)$ equals $\out(n) - \out(T_j) - \alpha n + \alpha T_j > \out(n) - \alpha n$. Thus, the difference for a smaller number $n-T_j$ exceeds the difference for a larger number $n$, which means $n$ cannot be a record. Hence, the Tribonacci representations of records must only have terms of the form $T_{P(t)}$ for some $t > 1$.

Moreover, suppose we have a number $n$ with Tribonacci representation $T_{P(2)} + T_{P(3)} + \cdots + T_{P(k-1)} + T_{P(k)}$ for some $k$. Consider a number $m < n$. Its Tribonacci representation can be built by removing some of the terms $T_{P(j)}$ for $j \leq k$ and adding terms of the form $T_{Q(i)}$. That is, we remove terms contributing a positive difference and add terms contributing a negative difference. Then $\out(n) - \alpha n > \out(m) - \alpha m$. So $n$ is a record.

Finally, we show that for $t>1$, any term $T_{P(t)}$ is contained in the Tribonacci representation of every sufficiently large record. First, as we just showed, for every $M$, the number $N_0 =T_{P(2)} + T_{P(3)} + \cdots + T_{P(M-1)} + T_{P(M)}$ is a record. Suppose $N > N_0$ does not contain $T_{P(t)}$ in its Tribonacci representation. Then $(\out(N_0) - \alpha N_0) - (\out(N) - \alpha N) \ge T_{P(t)+1} - \alpha T_{P(t)} - (T_{P(a_0)+1} - \alpha T_{P(a_0)} + T_{P(a_1)+1} - \alpha T_{P(a_1)} + \cdots)$, where the $a_i$ are all greater than $M$. But from Lemma~\ref{lemma:alphadiff}, we know that $T_{P(a_i)+1} - \alpha T_{P(a_i)} \le 2\frac{|\beta|^{P (a_i)}}{|\beta-\gamma|}$. Thus, $T_{P(a_0)+1} - \alpha T_{P(a_0)} + T_{P(a_1)+1} - \alpha T_{P(a_1)} + \cdots$ is bounded by $\frac{2}{|\beta-\gamma|}(|\beta|^{M+1} + |\beta|^{M+2} + \cdots) = \frac{2|\beta|^{M+1}}{|\beta-\gamma|(1-|\beta|)}$, which goes to 0 as $M$ goes to infinity. We know that $T_{P(t)+1} -\alpha T_{P(t)}$ is positive, so picking a sufficiently large $M$, we can ensure that $(T_{P(a_0)+1} - \alpha T_{P(a_0)} + T_{P(a_1)+1} - \alpha T_{P(a_1)} + \cdots)$ is smaller than $T_{P(t)+1} - \alpha T_{P(t)}$. Thus, $(\out(N_0) - \alpha N_0) - (\out(N) - \alpha N)$ is positive, and $N$ is not a record.

The argument for negative records is similar.
\end{proof}

Table~\ref{table:positivediff} displays the numbers $n$ that set a new record for a positive difference $T_{n+1}-\alpha n$. The first column shows $n$, the second column the approximate value of $|\alpha n-\out(n)|$, and the last column shows indices of Tribonacci numbers in the Tribonacci representation of $n$. The proposition above tells us that the numbers in the right column belong to sequence $P(k)$. Moreover, in the limit, the numbers approach the sequence $P(k)$ in reverse order. The data for negative records is in Table~\ref{table:negativediff}.

\begin{table}[ht!]
\begin{subtable}[c]{0.5\textwidth}
\centering
\begin{tabular}{|c|c|r|}   
1 	&0.1607132	&[3]\\
2 	&0.3214264	&[4]\\
3 	&0.4821397	&[4, 3]\\
10 	&0.6071324	&[6, 4, 3]\\
23 	&0.6964046	&[7, 6, 4, 3]\\
67 	&0.7677874	&[9, 7, 6, 4, 3]\\
148 	&0.7855602	&[10, 9, 7, 6, 4, 3]\\
341 	&0.8032164	&[12, 9, 7, 6, 4, 3]\\
422 	&0.8209892	&[12, 10, 9, 7, 6, 4, 3]\\
\end{tabular}
\subcaption{Records for positive difference.}
\label{table:positivediff}
\end{subtable}
\begin{subtable}[c]{0.5\textwidth}
\centering
\begin{tabular}{|c|c|r|}   
4 	&$-0.3571470$&[5]\\
28 	&$-0.5000291$&[8, 5]\\
177 	&$-0.5537556$&[11, 8, 5]\\
681 	&$-0.5542803$&[13, 11, 8, 5]\\
1104 	&$-0.5725777$&[14, 11, 8, 5]\\
1608 	&$-0.5731023$&[14, 13, 11, 8, 5]\\
4240 	&$-0.5758421$&[16, 14, 11, 8, 5]\\
4744 	&$-0.5763667$&[16, 14, 13, 11, 8, 5]\\
6872 	&$-0.5785818$&[17, 14, 11, 8, 5]\\
\end{tabular}
\subcaption{Records for negative difference.}
\label{table:negativediff}
\end{subtable}
\caption{Records for positive and negative difference.}
\end{table}

\section{Trithoff array}\label{sec:trithoffarray}

We want to build the analog of the Wythoff array for Tribonacci numbers. We call it the \textit{Trithoff} array. Our first column are numbers in increasing order whose Tribonacci representation ends in 1. For a number $n$ in the array, the number to the right is $\out(n)$. Thus, the second column are numbers whose Tribonacci representation ends in 10, and so on. Table~\ref{table:trithoffarray} shows the Trithoff array, where the last column is the sequence numbers in the OEIS \cite{OEIS} for corresponding rows. The array itself is in OEIS by antidiagonals as sequence A136175, and it is called the Tribonacci array. By analogy with Conway and Ryba \cite{ConwayRyba2016}, we will later extend the Trithoff array while we call this part the Garden State.

\begin{table}[ht!]
\begin{center}
\begin{tabular}{ccccccc|r}
 1 &  2 &  4 &  7 &  13 &  24 & \ldots & A000073\\
 3 &  6 & 11 & 20 &  37 &  68 & \ldots & A001590\\
 5 &  9 & 17 & 31 &  57 & 105 & \ldots & A000213\\
 8 & 15 & 28 & 51 &  94 & 173 & \ldots & A214899\\
10 & 19 & 35 & 64 & 118 & 217 & \ldots & A020992\\
12 & 22 & 41 & 75 & 138 & 254 & \ldots & A100683\\
\ldots & \ldots & \ldots & \ldots & \ldots & \ldots & \ldots &
\end{tabular}
\end{center}
\caption{The Trithoff array.}
\label{table:trithoffarray}
\end{table}

Here are some properties of the Trithoff array.

In the Trithoff array, the numbers in columns with index equaling 1, 2, or 3 mod 3, when sorted, form the positions of letters $a$, $b$, and $c$ in the Tribonacci word, respectively.

We denote the element in row $i$ and column $j$ as $T_{i,j}$. In particular, the first row is a shifted Tribonacci sequence: $T_{1,k} = T_{k+3}$.

Similar to extraFibs, we call integer sequences that follow the Tribonacci rule and end in positive integers \textit{extraTrib series} or \textit{extraTribs}. The sequences that just follow the Tribonacci rule we call \textit{Tribonacci-like} sequences.

It is well-known that any extraTrib can be expressed through Tribonacci numbers. For example, consider an extraTrib sequence $S_n$, where $S_0 = x$, $S_1 = y$, and $S_2 = z$. Then,
\[S_n = xT_{n-1} + y (T_{n+1} - T_n) + z T_n.\]

This allows one to derive many formulae connecting the values of different rows. For example

\begin{itemize}
\item $T_{2,n} = T_{1,n}+T_{1,n+1} = T_{1,n+3} - T_{1,n+2} = T_{1,n+2} - T_{1,n-1}$,
\item $T_{3,n} = T_{2,n-1}+T_{2,n} = T_{2,n+1} - T_{2,n-2} = T_{1,n} + T_{1,n+2} = T_{n-1} + 2T_n + T_{n+1}$.
\end{itemize}

Column 1 is sequence A003265:
\[1,\ 3,\ 5,\ 8,\ 10,\ 12,\ 14,\ 16,\ 18,\ 21,\ 23,\ 25,\ 27,\ 29,\ 32,\ 34,\ 36,\ 38,\ 40,\ \ldots.\]
We can estimate the growth rate of this sequence. Each next term is either 2 or 3 more than the previous term.
\begin{proposition}
The ratio of a number to its index in sequence A003265 above approaches 
\[\frac{\alpha^2+1}{2} = \frac{\alpha}{\alpha-1}\]
as the index approaches infinity.
\end{proposition}
This is correlated to the frequency of numbers ending in 1 in Tribonacci representation.

\begin{proof}
We want to see how many numbers that are less than $b$ end with a 1 in their Tribonacci representation. These numbers are either in the form of $w011$ or $w01$. Suppose $a$ is the number with Tribonacci representation $w$. Then $(\out(\out(\out(a))) + 3)_T = w011$ and $(\out(\out(a)) + 1)_T = w01$. Recall $\out(n) \approx \alpha n$. For the number with Tribonacci representation $w011$ to be less than $b$, we need $\alpha^3 a < b$ approximately. Thus the number of such $a$'s is about $\frac{b}{\alpha^3}$. Similarly, by the same logic for the second case, we need $\alpha^2 a < b$ approximately, and the number of such numbers is approximately $\frac{b}{\alpha^2}$. Thus, the number of numbers less than $b$ with Tribonacci representation ending in 1 is approximately $\frac{b}{\alpha^3} + \frac{b}{\alpha^2}$.

Suppose A003265$(r) = n$, or equivalently, the value in column 1 and row $r$ of the Trithoff array is $n$. This means that there are $r-1$ numbers less than $n$ with Tribonacci representations ending in 1. Using what we previously calculated, $r \approx \frac{n}{\alpha^3} + \frac{n}{\alpha^2}$. Isolating $n$, we get that $n \approx \frac{\alpha^3}{\alpha+1}r= \frac{\alpha}{\alpha-1}r$, which proves the proposition.
\end{proof}

\begin{corollary}
\label{cor:Trithoffformula}
A number in column $c$ and row $r$ of the Trithoff array can be estimated as 
\[\frac{r\alpha^c}{\alpha - 1}.\]
\end{corollary}

Column 2 is now sequence A353083:
\[2,\ 6,\ 9,\ 15,\ 19,\ 22,\ 26,\ 30,\ 33,\ 39,\ 43,\ 46,\ 50,\ 53,\ 59,\ 63,\ 66,\ 70,\ \ldots.\]
%2, 6, 9, 15, 19, 22, 26, 30, 33, 39, 43, 46, 50, 53, 59, 63, 66, 70, 74, 77, 83, 87, 90, 96, 100, 103, 107, 111, 114, 120, 124, 127, 131, 134, 140, 144, 147, 151, 155, 158, 164, 168, 171, 175, 179, 182, 188, 192, 195, 199, 202, 208, 212, 215, 219, 223, 226, 232, 236, 239, 245, 249, 252, 256, 260, 263, 269, 273, 276, 280, 283, 289, 293, 296, 300, 304, 307, 313, 317, 320, 324, 327, 333, 337, 340, 344, 348, 351, 357, 361, 364, 370, 374, 377, 381, 385, 388, 394, 398, 401 

Given a sequence $S_i$, the \textit{difference sequence} $D_i$ is defined as $D_i = S_{i+1} - S_i$. We can see that the difference sequence of an extraTrib is an extraTrib.

Suppose a row starts in $[a1]_T$. What is the row number?

Let us first introduce a new base $U$, based on the second row of the array. Let us denote $U_i = T_{2,i-2} = T_{i+1} - T_i$. Suppose we consider a Tribonacci-like representation system based on the second row of the array. Then using this system, $[\ldots a_4a_3a_2a_1a_0]_U$ is equal to $\cdots + a_4U_4 + a_3U_3 + a_2U_2 + a_1U_1 + a_0U_0 = \cdots + 6a_4 + 3a_3 + 2a_2 + a_1 + 0a_0$. Such representation is not unique, as $a_0$ doesn't affect the value. Also, $[1000]_U = [110]_U = 3$. However, the point of this system is provided by the following lemma.

\begin{lemma}
\label{lemma:rownumber}
If the row in Trithoff array starts with $[a1]_T$, then the row number is $[a1]_T-[a]_T = 1 + [a1]_U$.
\end{lemma}

\begin{proof}
The row number is the number of rows from the first to that row. Each row starts with a number whose Tribonacci representation ends in 1. There are $[a1]_T$ total numbers not exceeding $[a1]_T$. Out of those, the numbers of the form $[a0]_T$ correspond to numbers ending in zero, and there are exactly $[a]_T$ of them. Therefore, the total number of numbers not exceeding $[a1]_T$ and with Tribonacci representation ending in 1 is $[a1]_T-[a]_T = 1 + [a0]_T-[a]_T$. Using the fact that $U_i = T_{i+1} - T_i$, we get $[a1]_T-[a]_T  = 1 + [a0]_U = 1 + [a1]_U $.
\end{proof}

\subsection{Rows of the Trithoff array and their difference sequences}
\label{subsec:rowdiffsequences}

The difference sequence of a Fibonacci-like sequence is again the same sequence with the index shifted by 1. In the Tribonacci case, the situation is way more interesting.

The difference sequence of the first row of the Trithoff array is the second row. The difference sequence of the second row is the third row. The difference sequence of the third row is the seventh row. When we continue, we get the following sequence $a(n)$, which is now sequence A354215. In this sequence $a(n+1)$ is the row in the Trithoff array corresponding to the difference sequence of row $a(n)$.
\[1,\ 2,\ 3,\ 7,\ 19,\ 29,\ 81,\ 125,\ 353,\ 161,\ 1545,\ 705,\ 2001,\ \ldots ,\]
%1, 2, 3, 7, 19, 29, 81, 125, 353, 161, 1545, 705, 2001, 3089, 8769, 24897, 38433, 109121, 309825, 478273, 1357953, 2096257, 5951873, 2715905

The following proposition defines the difference sequence in terms of Tribonacci representation of the given sequence.

\begin{proposition}
The difference sequence of a row containing $a_T$ contains $|11_T \cdot a_T|$, where the multiplication is done in any integer base larger than 2.
\end{proposition}

\begin{proof}
We consider the difference sequences in terms of Tribonacci representation. First consider two large Tribonacci numbers: $(T_n)_T = 10^{n-3}$ and $(T_{n+1})_T = 10^{n-2}$. Then their difference is $T_{n+1} - T_n = T_{n-2} + T_{n-1}$. Thus, the Tribonacci representation is $(T_{n+1} - T_n)_T = (T_{n-2} + T_{n-1})_T = 110^{n-3}$. By linearity, if we take the difference sequence of a row containing $a_T$, the result contains $|11_T \cdot a_T|$, where the multiplication is done in any integer base larger than 2.
\end{proof}

For example, from row 1 containing 1, we get Tribonacci representation $11 \cot 1 =  11$, which is in row 2. Then multiplying 11 by 11, we get 121, which is evaluated to 9, which corresponds to row 3. Row 3 starts with 5, with Tribonacci representation 101. Repeating again, we multiply 101 from row 3 by 11 to get 1111 which evaluates to 14, corresponding to row 7.

Given a sequence $S_i$, the \textit{partial-sums sequence} $P_i$ is sequence defined as $P_i = \sum_{k=0}^iS_k$.

If a sequence $S_i$ has the difference sequence $D_i$, we call the sequence $S_i$ the \textit{difference-inverse} of $D_i$. It is well-known that any difference-inverse sequence equals a partial-sums sequence plus a constant.

Before proving our result about difference-inverses of extraTribs, we want to describe possible parities of Tribonacci-like integer sequences. One can check that there are four possible cases:

\begin{itemize}
\item Type (EEEE): All terms are even (for example, row 7);
\item Type (OOOO): All terms are odd (for example, row 3);
\item Type (EOEO): The terms alternate (for example, row 2);
\item Type (EEOO): The terms form a pattern of two even, then two odd (for example, row 1);
\end{itemize}

\begin{theorem}
An extraTrib always has a unique Tribonacci-like difference-inverse. The inverse is an extraTrib if and only if the given sequence does not belong to Type EEOO.
\end{theorem}

\begin{proof}
Take four adjacent terms $a$, $b$, $c$, $a+b+c$ of an extraTrib sequence $S$. Then the partial sums are $a$, $a+b$, $a+b+c$, and $2a+2b+2c$. Any difference-inverse sequence equals the partial sums sequence plus a constant. The sum of the first three is $3a+2b+c$, so we need to add $\frac{c-a}{2}$ to each term to make the fourth term the sum of the first three terms. Consider a Tribonacci-like sequence $Q$ that starts with four consecutive entries
\[\frac{a+c}{2}, \quad \frac{a+2b+c}{2}, \quad \frac{a+2b+3c}{2}, \quad \frac{3a+4b+5c}{2}.\]
As $Q$ follows the Tribonacci rule, its difference sequence follows the Tribonacci rule, and since the difference sequence agrees with $S$ for three consecutive terms, it agrees with $S$ everywhere.

But the new sequence $Q$, though it is Tribonacci-like, is not always an extraTrib, as it is not guaranteed to be an integer sequence. When we add $\frac{c-a}{2}$ to the terms, the new terms are integers if and only if $a$ and $c$ have the same parity. Thus Tribonacci-like sequences with all the terms of the same parity and the ones where parity alternates have integral difference-inverses, while the sequences of Type EEOO do not.
\end{proof}

We call an extraTrib $S$ \textit{invertible} if there exists an extraTrib with the difference sequence $S$. Equivalently, an extraTrib is invertible if it is of types EEEE, OOOO, or EOEO.

Testing the rows in the Trithoff array, we get the sequence of invertible rows, which is now sequence A353178: 
\[2,\ 3,\ 4,\ 7,\ 11,\ 12,\ 16,\ 17,\ 19,\ 20,\ 21,\ 25,\ 26,\ 28,\ 29,\ 30,\ 33,\ 34,\ \ldots.\]
%2, 3, 4, 7, 11, 12, 16, 17, 19, 20, 21, 25, 26, 28, 29, 30, 33, 34, 35, 38, 42, 43, 45, 46, 47, 50, 51, 52, 55, 59, 60, 61, 64, 68, 69, 73, 74, 76, 77, 78, 81, 82, 83, 86, 90, 91, 92, 95, 99, 100, 104, 105, 107, 108, 109, 112, 116, 117, 121, 122, 124, 125, 126, 130, 131, 133, 134, 135, 138, 139, 140, 143, 147, 148, 149, 152, 156, 157, 161, 162, 164, 165, 166, 169, 173, 174, 178, 179, 181, 182, 183,

Non-invertible rows are now sequence A353193: 
\[1,\ 5,\ 6,\ 8,\ 9,\ 10,\ 13,\ 14,\ 15,\ 18,\ 22,\ 23,\ 24,\ 27,\ 31,\ 32,\ 36,\ 37,\ 39,\ \ldots.\]
%1, 5, 6, 8, 9, 10, 13, 14, 15, 18, 22, 23, 24, 27, 31, 32, 36, 37, 39, 40, 41, 44, 48, 49, 53, 54, 56, 57, 58, 62, 63, 65, 66, 67, 70, 71, 72, 75, 79, 80, 84, 85, 87, 88, 89, 93, 94, 96, 97, 98, 101, 102, 103, 106, 110, 111, 113, 114, 115, 118, 119, 120, 123, 127, 128, 129, 132, 136, 137, 141, 142, 144, 145, 146, 150, 151, 153, 154, 155, 158, 159, 160, 163, 167, 168, 170, 171, 172, 175, 176, 177, 180,

Interestingly, nothing of this sort appears in the Fibonacci case. As the difference sequence of a Fibonacci-like sequence is the sequence itself (shifted), we get that all extraFibs are invertible, as they invert to themselves.

\begin{theorem}
If we start with an extraTrib and keep inverting it, we will get to a non-invertible sequence in a finite number of steps.
\end{theorem}

\begin{proof}
We start by looking at how the difference operator changes types. Type EEOO changes to EOEO; the latter changes to OOOO, then to EEEE. The sequence of type EEEE stays EEEE.

Define the valuation of an extraTrib to be the largest integer $n$ such that all numbers in the sequence are multiples of $2^n$. An extraTrib of valuation $n$, after dividing by $2^n$, is of the type OOOO, EOEO, or EEOO. Similar to before, the difference sequence of sequence $2^n$ times type EEOO is sequence $2^n$ times type EOEO. The latter changes to $2^n$ times type OOOO, and the difference of that has valuation $m > n$. After taking at most three difference operators, the result has a greater valuation than the starting sequence. Thus, for every three inverse differences, the valuation decreases, and the result follows.
\end{proof}

\subsection{The difference sequence for the first column}
\label{subsec:columndiffsequences}

Consider the first column (sequence A003265):
\[1,\ 3,\ 5,\ 8,\ 10,\ 12,\ 14,\ 16,\ 18,\ 21,\ 23,\ 25,\ 27,\ 29,\ 32,\ 34,\ 36,\ 38,\ 40,\ \ldots,\]
and its difference sequence:
\[2,\ 2,\ 3,\ 2,\ 2,\ 2,\ 2,\ 2,\ 3,\ 2,\ 2,\ 2,\ 2,\ 3,\ 2,\ 2,\ 2,\ 2,\ 2,\ \ldots,\]

\begin{proposition}
The difference sequence of the first column consists of twos and threes. The indices of the rows such that the next value in the first column is 3 greater form sequence A305373.
\end{proposition}

\begin{proof}
Consider a number $a$ in the first column. Looking at its Tribonacci representation that must end in a 1, consider cases $b001$, $b011$, $b0101$, and $b01101$. 

If $[a]_T = b0101$, then the next term in column 1 will be $|b1001| = [b0101]_T + 3 = a+3$, so the difference is 3.

For other possibilities, $b001$, $b011$, and $b01101$, the next terms are $b011$, $|b101|$, and $|b10001|$ correspondingly. In all these cases, the next term is increased by 2.

The indices of the rows such that the next value in the first column is 3 greater correspond to rows starting with numbers of the form $[b0101]_T$. The row number is $[b0101]_T-[b010]_T$ by Lemma~\ref{lemma:rownumber}. This equals $\out(b)+\out^2(b)+3$. Sequence A305373 is defined as the sum of A003144 and A003145. A003144 is the sequence $[x0]_T+1$ and A003145 is the sequence $[x01]_T+1$. Thus, their sum is $\out(x)+\out^2(x)+3$. Thus, our sequence is A305373.
\end{proof}

\section{Extending the Trithoff array. Precolumns}\label{sec:precolumns}

We extend the Trithoff array to the left by using the rule that in every row, each number is the sum of the three previous numbers. We assume that the columns of the Trithoff array start with index 1. Similar to Conway and Ryba \cite{ConwayRyba2016}, we call the 0th column \textit{the wall}, column $-1$, \textit{the seed}, and column $-2$, \textit{the pre-seed}. Table~\ref{table:Trithoffpre} shows the upper-left part of the Trithoff array with precolumns.

\begin{table}[H]
\centering
\begin{tabular}{|l|l|l||r|r|r|r|}
\hline
$i=-2$ (pre-seed) & $i=-1$ (seed) & $i=0 $ (wall) & $i=1$ & $i=2$ & $i=3$ & $i=4$ \\ \hline
0 & 0 & 1 & 1 & 2 & 4 & 7 \\ \hline
0 & 1 & 2 & 3 & 6 & 11 & 20 \\ \hline
1 & 1 & 3 & 5 & 9 & 17 & 31 \\ \hline
1 & 2 & 5 & 8 & 15 & 28 & 51 \\ \hline
1 & 3 & 6 & 10 & 19 & 35 & 64 \\ \hline
2 & 3 & 7 & 12 & 22 & 41 & 75 \\ \hline
\end{tabular}
\caption{Trithoff array with precolumns.}
\label{table:Trithoffpre}
\end{table}

Before describing precolumns, we need to do some work with Tribonacci representations.

We call a Tribonacci representation \textit{non-canonical} if it consists of integers zero and one, but might contain three consecutive ones, as opposed to the Tribonacci representation, which we might call \textit{canonical} to emphasize it. For example, suppose $n=13$. Then its canonical representation is 10000. While a non-canonical word 1110 also evaluates to 13. Similar to Conway and Ryba \cite{ConwayRyba2016}, we denote by $|v|$ the canonization of a non-canonical representation $v$.

Suppose we have a non-canonical Tribonacci representation of the number $n$. We can view this representation as a sum of distinct Tribonacci numbers. We call replacing $T_{n-2}+T_{n-1}+T_n$ in this sum with $T_{n+1}$ \textit{carrying}.

\textbf{Canonization of a non-canonical representation.} Suppose we have a non-canonical representation $v$ of number $n$. We use the leftmost possible carry. In other words, we are replacing the leftmost $0111$ with $1000$. This way, the new representation contains only digits zero and one and evaluates to the same number $n$, while the sum of digits decreases. That means the procedure terminates in a canonical representation of $n$.

\begin{lemma}\label{lemma:outnoncanonical}
Suppose a binary word $w$ is a non-canonical representation of integer $n$, then $w0$ is a non-canonical representation of $\out(n)$.
\end{lemma}

\begin{proof}
Consider a canonization procedure described above for binary words $w$ and $w0$. Each step of this procedure is the same, except we have 0 at the end of the second word. Thus $|w0| = |w|0 = (n)_T0 = (\out(n))_T$.
\end{proof}

Now we are ready to describe the Tribonacci representation of precolumns, given the Tribonacci representation of the first column.

\begin{lemma}
\label{lemma:precolumns}
Suppose number $n$ in column 1 has the Tribonacci representation $abc1$, where $b$ and $c$ are digits zero or one, and $a$ is a binary word. Then the wall $w$ in the same row equals $[abc]_T+1$, the seed in the same row equals $[ab]_T+c$, and the pre-seed in the same row equals $[a]_T+b$.
\end{lemma}

\begin{proof}
We have $(\out(n))_T = abc10$ and $(\out(\out(n)))_T = abc100$. By definition, the corresponding wall element $w$ is $\out(\out(n)) - \out(n)-n$. We can split $n$ as $[abc0]_T + 1$, $\out(n)$ as $[abc00]_T + 2$ and $\out(\out(n))$ as $[abc000]_T + 4$. Then $w = ([abc000]_T + 4) - ([abc00]_T + 2) - ([abc0]_T + 1) = ([abc000]_T - [abc00]_T - [abc0]_T) + 1 = [abc]_T + 1$.

The seed $s$ equals $\out(n) - n - w = [abc10]_T - [abc1]_T - [abc]_T - 1 = ([ab000]_T + 4c + 2) - ([ab00]_T + 2c + 1) - ([ab0]_T + c) - 1 = ([ab000]_T - [ab00]_T - [ab0]_T) + c = [ab]_T + c$.

The pre-seed $p$ equals $n - w - s = [abc1]_T - ([abc]_T + 1) - ([ab]_T + c)= ([a000]_T + 4b + 2c + 1) - ([a00]_T + 2b + c + 1) - ([a0]_T + b + c) = ([a000]_T - [a00]_T - [a0]_T) + b= [a]_T + b$.
\end{proof}

Analogs for our case for the structure of the Tribonacci Garden State of Facts 5T describe how the next term to the right depends on the previous term.

\begin{fact}{5Ta}
A garden term $n$ is followed by $\out(n)$.
\end{fact}

This is true by definition.

\begin{fact}{5Tc}
A wall term $w$ is followed by $\out(w) - 1$.
\end{fact}

\begin{proof}
Suppose the wall term $w$ is followed by the garden term $n$. Suppose the Tribonacci representation of $n$ is $abc1$. Then from Lemma~\ref{lemma:precolumns} we have $w = [abc]_T + 1$. Thus, $n = \out(w-1) + 1$.

We consider cases as we can assume that $b$ and $c$ are not both 1.
\begin{itemize}
\item If $bc$ is 00, then $\out(w) - 1 = \out([abc]_T+1)-1 = \out([a00]_T+1)-1= \out([a01]_T)-1= [a010]_T-1= [a001]_T= [abc1]_T$. \item If $bc$ is 01, then $\out(w) - 1 = \out([abc]_T+1)-1 = \out([a01]_T+1)-1 = \out([a10]_T)-1$. The representation $[a10]_T$ might be non-canonical, so by Lemma~\ref{lemma:outnoncanonical} $\out(w) - 1 = [a100]_T-1=[a011]_T=[abc1]_T$.
\item If $bc$ is 10, then $\out(w) - 1 = \out([abc]_T+1)-1=\out([a10]_T+1)-1=\out([a11]_T)-1$. Again $[a11]_T$ might be a non-canonical representation, and by Lemma~\ref{lemma:outnoncanonical} we have $\out([a11]_T)-1 = [a110]_T-1 = [a101]_T = [abc1]_T$. 
\end{itemize}
So $w$ is always followed by $out(w)-1$.
\end{proof}

The analog of Fact 5b describes how to calculate the seed from the pre-seed and the wall from the seed in the same row. The formulae depend on the last digits of the Tribonacci representation of the first garden term in the same row. Consider row $x$, where we denote the pre-seed by $p(x)$, the seed by $s(x)$, the wall term as $w(x)$, and the first garden term as $g(x)$. 

\begin{fact}{5Tb}
If $(g(x))_T$ ends with 11, then $w(x) = \out(s(x))$ and $s(x) = \out(p(x))+1$. If $(g(x))_T$ ends in $001$, then $w(x) = \out(s(x))+1$ and $s(x) = \out(p(x))$. If $(g(x))_T$ ends  in $101$, then $w(x) = \out(s(x))+1$ and $s(x) = \out(p(x))-1$.
\end{fact}

\begin{proof}
We use Lemma~\ref{lemma:precolumns} that states that if $(g(x))_T = abc1$, then $w(x) = [abc]_T+1$, $s(x) = [ab]_T+c$, and $p(x) = [a]_T+b$.

If $(g(x))_T$ ends with 11, then $b=0$ and $c = 1$. We have $w(x) = [a01]_T + 1 = [a10]_T$ and $s(x) = [a0]_T + 1 = [a1]_T$. The representation $[a1]_T$ might be non-canonical, but it still respects the out function. Thus, in this case, $w(x) = \out(s(x))$. We also have $p(x) = [a]_T + 0 = [a]_T$, so $s(x) = \out(p(x)) + 1$.

If $(g(x))_T$ ends in 01, then $c = 0$. Thus, $w(x) = [ab0]_T+1$ and $s(x) = [ab]_T + 0 = [ab]_T$, and therefore $w(x) = \out(s(x))+1$.

If $(g(x))_T$ ends in 001, then $b = c = 0$. Thus, $s(x) = [a0]_T + 0 = [a0]_T$ and $p(x) = [a]_T + 0 = [a]_T$, and therefore $s(x) = \out(p(x))$.

If $(g(x))_T$ ends in 101, then $b = 1$ and $c = 0$, so $s(x) = [a1]_T + 0 = [a0]_T + 1$ and $p(x) = [a]_T + 1$. To prove that $[a0]_T + 1 = \out([a]_T + 1) - 1$ we must consider two cases. If $a$ ends in 0, then adding 1 might result in a non-canonical representation, but we still can take the successor. In this case, $\out(p(x)) = \out([a]_T + 1) = [a0]_T + 2 = s(x) + 1$, so the lemma holds. If $a$ ends in 1, it must end in 01, otherwise, the representation of $(g(x))_T$ would not be canonical. Let $a = d01$ for some binary string $d$. Then, $[a]_T + 1 = [d01]_T + 1 = [d10]_T$. Again, $d10$ might not be canonical. In any case, $\out(p(x)) = \out([a]_T + 1) = \out([d10]_T) = [d100]_T = \out([d01]_T) + 2 = \out([a]_T) + 2 = s(x) + 1$, and the lemma still holds.
\end{proof}

\subsection{Precolumns}

The following theorem describes precolumns in terms of the Tribonacci word.

\begin{theorem}
\begin{itemize}
\item The wall is an increasing sequence of numbers that are positions of letters $a$ and $b$ in the Tribonacci word.
\item The seed is a non-decreasing sequence, starting with 0 followed by all integers, where the positions of letter $a$ in the Tribonacci word are doubled, and all the other integers are not doubled.
\item The pre-seed forms a non-decreasing sequence, starting with two zeros followed by all integers, where the positions of letters $a$ and $b$ in the Tribonacci word are tripled, while the positions of letters $c$ are doubled.
\end{itemize}
\end{theorem}

\begin{proof}
Consider a term $x$ in the first column of the Trithoff array in the form $[abc1]_T$. By Lemma~\ref{lemma:precolumns}, the wall $w$ in the same row equals $[abc]_T+1$, the seed in the same row equals $[ab]_T+c$, and the pre-seed in the same row equals $[a]_T+b$.

\textbf{The wall.} The word $abc$ goes, in order, through all Tribonacci representations ending in 0 or 01. Thus, the wall term goes in order through numbers that are one greater than numbers with Tribonacci representations ending in 0 or 01. So by Theorem~\ref{thm:abcCorrespondsTo0-01-11} the wall term will go through all positions of the letters $a$ and $b$ in the Tribonacci word.

\textbf{The seed.} Consider how it changes when moving to the previous row. 

Suppose $b = 0$ and $c = 1$. Then the first column term $x$ for that row is $[a011]_T$, and the corresponding seed is $[a0]_T + 1$, matching a position of letter $a$ in the Tribonacci word. The previous row has the garden term $[a001]_T$, and the corresponding seed equals $[a0]_T$. Thus the seed increases by 1 from the previous row.

Suppose $b = 1$ and $c = 0$. Then $x = [a101]_T$, and the corresponding seed is $[a1]_T$, matching a position of letter $a$ in the Tribonacci word. The previous row has garden term $[a011]_T$, and the corresponding seed is $[a0]_T + 1 = [a1]_T$. Thus this row's seed equals the previous seed.

Suppose $b = c = 0$. Then $x = [a001]_T$, and the corresponding seed is $[a0]_T$, which matches a position of letter $b$ or $c$ in the Tribonacci word. Suppose $a'$ is the Tribonacci representation of $[a]_T -1$. The previous row has garden term $[a'101]_T$ or $[a'011]_T$ with the same corresponding seed $[a'1]_T$ or $[a'0]_T + 1$, which is less than the current seed $[a0]_T$ by 1.

To summarize, a row's seed equals the previous seed when $b = 1$ and $c = 0$; otherwise, the row's seed is one greater than the previous seed. Thus, only the positions of the letter $a$ in the Tribonacci word are doubled.

\textbf{The pre-seed.} The pre-seeds are non-decreasing because as $[abc1]_T$ increases, $[a]_T + b$ either remains the same or increases by 1.

The number $[a]_T$ appears at least twice as a pre-seed: in rows with the garden value equal $[a001]_T$ or $[a011]_T$.

The other possible values for the first column are $[a101]_T$ where $a$ does not end in two 1s, and the corresponding pre-seed is $[a]_T+1$. If $[a]_T=x0$, for some prefix $x$, then such pre-seed has the possibly non-canonical representation $[x1]_T$, so its canonical representation has the number of trailing zeroes equal to 0 mod 3 and if $[a]_T=x01$ then such pre-seed has the possibly non-canonical representation $[x10]_T$, which has a canonical representation with a number of trailing zeroes equal to 1 mod 3. These extra values are indices of $a$ and $b$ in the Tribonacci word.
\end{proof}

Column $0$, the wall, is now sequence A353084:
\[1,\ 2,\ 3,\ 5,\ 6,\ 7,\ 8,\ 9,\ 10,\ 12,\ 13,\ 14,\ 15,\ 16,\ 18,\ 19,\ 20,\ \ldots.\]
%1, 2, 3, 5, 6, 7, 8, 9, 10, 12, 13, 14, 15, 16, 18, 19, 20, 21, 22, 23, 25, 26, 27, 29, 30, 31, 32, 33, 34, 36, 37, 38, 39, 40, 42, 43, 44, 45, 46, 47, 49, 50, 51, 52, 53, 54, 56, 57, 58, 59, 60, 62, 63, 64, 65, 66, 67, 69, 70, 71, 73, 74, 75, 76, 77, 78, 80, 81, 82, 83, 84, 86, 87, 88, 89, 90, 91, 93, 94, 95, 96, 97, 99,

Column $-1$, the seed, is now sequence A353086:
\[0,\ 1,\ 1,\ 2,\ 3,\ 3,\ 4,\ 5,\ 5,\ 6,\ 7,\ 7,\ 8,\ 8,\ 9,\ 10,\ 10,\ \ldots,\]
%0, 1, 1, 2, 3, 3, 4, 5, 5, 6, 7, 7, 8, 8, 9, 10, 10, 11, 12, 12, 13, 14, 14, 15, 16, 16, 17, 18, 18, 19, 20, 20, 21, 21, 22, 23, 23, 24, 25, 25, 26, 27, 27, 28, 29, 29, 30, 31, 31, 32, 32, 33, 34, 34, 35, 36, 36, 37, 38, 38, 39, 40, 40, 41, 42, 42, 43, 44, 44, 45, 45, 46, 47, 47, 48, 49, 49, 50, 

The numbers that are not doubled in the seed are now sequence A351631:
\[0,\ 2,\ 4,\ 6,\ 9,\ 11,\ 13,\ 15,\ 17,\ 19,\ 22,\ 24,\ 26,\ 28,\ 30,\ 33,\ 35,\ \ldots.\]
%0, 2, 4, 6, 9, 11, 13, 15, 17, 19, 22, 24, 26, 28, 30, 33, 35, 37, 39, 41, 43, 46, 48, 50, 53, 55, 57, 59, 61, 63, 66, 68, 70, 72, 74, 77, 79, 81, 83, 85, 87, 90, 92, 94, 96, 98, 100, 103, 105, 107, 109, 111, 114, 116, 118, 120, 122, 124, 127, 129, 131, 134, 136, 138, 140, 142, 144, 147, 149, 151, 153, 155, 158, 160, 162, 164, 166, 168, 171, 173, 175, 177, 179, 182, 184, 186, 188, 190, 192, 195, 197, 199,

Column $-2$, the pre-seed, is now sequence A353090:
\[0,\ 0,\ 1,\ 1,\ 1,\ 2,\ 2,\ 2,\ 3,\ 3,\ 3,\ 4,\ 4,\ 5,\ 5,\ 5,\ 6,\ 6,\ 6,\ \ldots .\]
%0, 0, 1, 1, 1, 2, 2, 2, 3, 3, 3, 4, 4, 5, 5, 5, 6, 6, 6, 7, 7, 7, 8, 8, 8, 9, 9, 9, 10, 10, 10, 11, 11, 12, 12, 12, 13, 13, 13, 14, 14, 14, 15, 15, 15, 16, 16, 16, 17, 17, 18, 18, 18, 19, 19, 19, 20, 20, 20, 21, 21, 21, 22, 22, 22, 23, 23, 23, 24, 24, 25, 25, 25, 26, 26, 26, 27, 27, 27, 28, 28, 29, 29, 29, 30, 30, 30, 31, 31, 31, 32, 32, 32, 33, 33, 33, 34, 34, 34, 35,

\subsection{Fact 6}

We can combine the results in this section into the following analog of Fact 6.

\begin{fact}{6T}
Each positive integer appears once in the Trithoff array. In addition, positions of $a$ in the Tribonacci word appear 1 time in the wall, 2 times in the seed, and 3 times in the pre-seed. Positions of $b$ in the Tribonacci word appear 1 time in the wall, 1 time in the seed, and 3 times in the pre-seed. Positions of $c$ in the Tribonacci word appear 0 times in the wall, 1 time in the seed, and 2 times in the pre-seed.
\end{fact}

\section{ExtraTribs and their multiples}\label{sec:multiples}

According to Fact 7 from Conway-Ryba paper \cite{ConwayRyba2016}, every series that satisfies Fibonacci's rule and is eventually positive is represented in the Garden State. Namely, every extraFib has a tail that is a row in the Wythoff array. We want to generalize this to extraTribs.

As an example, consider the sequence that is twice the Tribonacci numbers: 0, 0, 2, 2, 4, 14, 26, 48, and so on. The first few terms can be found in the first row of the Trithoff array. The tail starting with 14, 26, and 48 is the seventh row of the Trithoff array.

To help us deal with extraTribs, we need to deal with improper Tribonacci representations discussed below.

\subsection{Improper Tribonacci representation and its canonization}

We call a Tribonacci representation \textit{improper} if it uses digits other than zero and one. To emphasize the difference, we call a canonical or a non-canonical representation (aka representations that use only ones and zeros) \textit{proper}. Suppose a word $v$ is an improper Tribonacci representation of integer $n$. As before, we denote by $|v|$ the canonization of the word $v$, aka the canonical Tribonacci representation of $n$. In other words, $|v| = (n)_T$. An example of an improper representation of 13 is 1030, and its canonization is 10000.

Suppose we have an improper Tribonacci representation of number $n$, which is a linear combination of Tribonacci numbers. Recall that we call replacing $T_{n-2}+T_{n-1}+T_n$ with $T_{n+1}$ carrying. We call replacing $T_{n+1}$ with $T_{n-2}+T_{n-1}+T_n$ \textit{reverse carrying}. In terms of a Tribonacci representation of $n$, the carrying replaces $dabc$ with $(d+1)(a-1)(b-1)(c-1)$ for $a,b,c > 0$, while reverse carrying replaces $dabc$ with with $(d-1)(a+1)(b+1)(c+1)$ for $d > 0$.

Our goal in this section is to introduce the canonization procedure that, given enough zeros at the end of a Tribonacci representation, converts an improper representation of some number to a canonical representation of the same number in a finite number of steps.

We call the position (index) of the leftmost digit that is greater than 1 in an improper representation $w$ the \textit{improper boundary index} and the value at this position the \textit{improper boundary value}. Next, we define the \textbf{weight} of $w$ to be the sum of all digits in $w$ that are to the right from the last 0 preceding the improper boundary.

We look at the word $w$ from left to right, where we can assume that $w$ is padded with zeros on the left if needed. 

\begin{lemma}
\label{lemma:leftomostboundary}
The leftmost carrying does not move the improper boundary index to the left and does not increase the weight.
\end{lemma}

\begin{proof}
The leftmost carrying replaces $0abc$ with $1(a-1)(b-1)(c-1)$, where $a,b,c > 0$. Thus the boundary does not move to the left. If the leftmost carrying acts on digits to the left of the 0 preceding the boundary, then it does not change the weight, if not, it decreases it.
\end{proof}

Consider the following procedure acting on an improper representation of a number. \textbf{Canonization procedure:}
\begin{itemize}
\item Step 1. Use leftmost carrying when possible. This procedure ensures that the longest prefix of the word $w$ that contains only 0's and 1's is in the canonical form. For the following steps, we can always assume that no three consecutive digits are greater than 0.
\item Step 2. When step 1 is not available, work on the leftmost improper boundary $a$. There are two cases of what we do depending on what is before $a$: 0, or 01.
\begin{itemize}
\item Step 2a. Replace $0abcd$ with $1(a-2)bc(d+1)$. This is a combination of reverse carrying (replacing $0abcd$ with $0(a-1)(b+1)(c+1)(d+1)$) and carrying (replacing $0(a-1)(b+1)(c+1)(d+1)$ with $1(a-2)bc(d+1)$).
\item Step 2b. Replace $01a0cd$ with $10(a-2)0(c+1)(d+1)$. This is a combination of reverse carrying (replacing $01a0cd$ with $01(a-1)1(c+1)(d+1)$) and carrying (replacing $01(a-1)1(c+1)(d+1)$ with $10(a-2)0(c+1)(d+1)$).
\end{itemize}
\end{itemize}

These operations are not defined when $a$ is one of the last three digits of a number. If the procedure on the word $w'$ ends with a canonical word $w$ we call the $w$ the canonization of $w'$: $w = |w'|$. If $d$ is an integer, we write the string consisting of $m$ copies of $d$ as $d^m$.

Note that when our presentation is non-canonical but proper, we only need Step 1.

\begin{theorem}
\label{thm:canon}
Given an improper representation $w$ of an integer $n$ with weight $m$ ending in at least $3m$ zeros, the canonization procedure applied to $w$ terminates in the finite number of steps.
\end{theorem}

\begin{proof}
Each of the steps in canonization does not change the value of the number while making it lexicographically larger. In addition, Step 1 decreases the sum of digits of the Tribonacci representation, while both Steps 2a and 2b do not change the sum of the digits.

Now we look at the weight. Step 1 does not increase the weight. Step 2b decreases the weight. Now we look at Step 2a. If $a>3$, then Step 2a does not change the weight, but the next operation has to be either Step 1 or Step 2b, both of which decrease the weight. If $a = 3$, then Step 2a replaces $03bcd$ with $11bc(d+1)$. If $b=0$, the boundary moves, and the weight decreases; if $b > 0$, we perform Step 1 that decreases the weight. If $a=2$, then Step 2a replaces $02bcd$ with $10bc(d+1)$, and the weight decreases.

Step 1 does not change the number of trailing zeros while increasing the lexicographical order. Moreover, the digits of the improper representation of a number $x$ cannot exceed $x$. It follows that we can only make a finite number of such steps in a row. The total number of times Step 2 runs is also finite as Steps 2b decrease the weight, and each Step 2a is followed by steps decreasing the weight.

Now we estimate the number of trailing zeros that we need. Notice that though Step 2a might not decrease the weight, as soon as the boundary moves, the weight is decreased. For every operation in Step 2, we need three digits after the boundary to be available. It follows that it is enough to have $3m$ trailing zeros.
\end{proof}

The canonization procedure for the word $w$ uses a fixed number of zeros. If we add more zeros to the end of $w$, the procedure is still the same.

\begin{corollary}
\label{cor:canon}
If the canonization procedure above for an improper representation $w$ of number $n$ terminates in the canonical Tribonacci representation $|w|$, then the same procedure for $w0^k$ terminates in $|w|0^k$.
\end{corollary}

\subsection{ExtraTribs}

Consider an extraTrib $S_n$ that starts with non-negative numbers $a$, $b$, and $c$, that is $S_0 = a$, $S_1 = b$, and $S_2 = c$. Then $S_n$ is a linear combination of three sequences:
\begin{itemize}
\item A sequence that starts as 1, 0, 0. This sequence continues as 1, 1, 2 and can be described as shifted Tribonacci sequence $T_{n-1}$.
\item A sequence that starts as 0, 1, 0. This sequence continues as 1, 2, 3, 6, and so on. It is a second row of the Trithoff array and can be represented as the sequence $T_n+T_{n+1}$.
\item A sequence that starts as 0, 0, 1, which is the Tribonacci sequence $T_n$.
\end{itemize}
Thus, 
\[S_n = aT_{n-1}+b(T_n+T_{n+1})+cT_n= aT_{n-1}+(b+c)T_n+bT_{n+1}.\]

\begin{fact}{7T}
Any extraTrib sequence has its tail appearing in the array.
\end{fact}

\begin{proof}
Suppose we are given an extraTrib sequence $S_n$. Its terms can be expressed as a positive integer linear combination of shifted Tribonacci sequences: $S_n = aT_{n-1}+(b+c)T_n+bT_{n+1}$. Thus, for $n > 3$, the term $S_n$ has an improper Tribonacci representation $b(b+c)a0^{n-4}$. By Theorem~\ref{thm:canon}, there exists $N_0$, such that the canonization procedure for the word $b(b+c)a0^{N_0-4}$ terminates in a word $w$. By Corollary~\ref{cor:canon} for any number $N \geq N_0$, we have $(S_n)_T = w0^{N-N_0}$. Thus, all these numbers are in the same row of the Trithoff array.
\end{proof}

\subsection{Multiples of Tribonacci sequences in the array}

According to Fact 8, any positive multiple of an extraFibs is an extraFib. It immediately generalizes to Fact 8T.

\begin{fact}{8T}
Any positive multiple of an extraTribs is an extraTribs.
\end{fact}

Thus, any multiple of an extraTribs appears in the Trithoff array.

We wrote a program to calculate multiples of the Tribonacci numbers and find them in the Trithoff array. This data is summarized in Table~\ref{table:TribonacciMultiples}. The first row is the multiple coefficient. The next row of the table is the row of the array where the tail of the $n$th multiple appears. The third row of the table is the value of the first column of that row in the array. The last row of the table is the third row divided by the first row.

\begin{table}[ht!]
\begin{center}
    \begin{tabular}{c|ccccccccccccc}
         multiple &1&2&3&4&5&6&7&8&9& 10 & 11 & 12 & \dots\\
         row \#   &1&7&10&81&101&121&141&161&1126& 1251 & 1376 & 1501 & \dots\\
         first column &1&14&21&176&220&264&308&352&2466& 2740 & 3014 & 3288 & \dots\\
         Trib \# &1&7&7&44&44&44&44&44&274& 274 & 274 & 274 & \dots\\
    \end{tabular}
\end{center}
\caption{Tribonacci multiples}
\label{table:TribonacciMultiples}
\end{table}

The sequence of row numbers is now sequence A351685:
\[1,\ 7,\ 10,\ 81,\ 101,\ 121,\ 141,\ 161,\ 1126,\ 1251,\ 1376,\ 1501,\ 1626,\ 1751,\ \ldots.\]
%1, 7, 10, 81, 101, 121, 141, 161, 1126, 1251, 1376, 1501, 1626, 1751, 1876, 2001, 2126, 2251, 2376, 2501, 2626, 17117, 17895, 18673, 19451, 20229, 21007, 21785, 22563, 23341, 24119, 24897, 25675, 26453, 27231, 28009, 28787, 29565, 30343, 31121, 31899, 32677, 33455, 34233, 35011, 35789, 36567, 37345, 128969,

The numbers that start off the rows that are multiples of the Tribonacci sequence are now sequence A351689:
\[1,\ 14,\ 21,\ 176,\ 220,\ 264,\ 308,\ 352,\ 2466,\ 2740,\ 3014,\ 3288,\ 3562,\ \ldots.\]
%1, 14, 21, 176, 220, 264, 308, 352, 2466, 2740, 3014, 3288, 3562, 3836, 4110, 4384, 4658, 4932, 5206, 5480, 5754, 37510, 39215, 40920, 42625, 44330, 46035, 47740, 49445, 51150, 52855, 54560, 56265, 57970, 59675, 61380, 63085, 64790, 66495, 68200, 69905, 71610, 73315, 75020, 76725, 78430, 80135, 81840, 282632,

Notice that this sequence contains runs of arithmetic progressions. For example, numbers 176, 220, 264, 308, 352 form an arithmetic progression with difference 44. Correspondingly, number 44 appears 5 times in row 4 of Table~\ref{table:TribonacciMultiples}. The next 13 numbers form an arithmetic progression with difference 274. The next 27 numbers form an arithmetic progression with difference 1705. 

%The multiples of the second row appear as rows: 2, 19, 28, 125, 286, 343, 400, 1545, 1738, 1931, 2124, 2317, 2510, 2703, 2896, 3089, 3282, 21619,  22820, 24021, 25222, 26423, 27624, 28825, 30026, 31227, 32428, 33629, 34830, 36031, 37232, 38433, 39634, 40835, 42036, 43237, 44438, 45639, 46840, 48041, 49242, 50443, 51644, 52845, 54046, 55247, 56448, 57649, 58850

\subsection{Order of multiples}

Fact 9, states that multiples of any extraFib series appear in order in the Wythoff array.

\begin{fact}{9T}
Multiples of extraTribs appear in order in the Trithoff array.
\end{fact}

\begin{proof}
Let the sequence in the array be $S_n$. Suppose we found $kS_n$ in the array. If it requires $m$ trailing zeroes to canonize, then the first term in the Trithoff array corresponding to $kS_n$ starts with the canonization of $kS_{m+1}$, as $S_{m+1}$ is the smallest term in the array $S_n$ which has $m$ zeros. Now we want to canonize $(k+1)S_n=kS_n+S_n$. First, canonizing $kS_n$ requires $m$ trailing zeroes, then adding $S_n$ requires a total of $m^\prime\ge m$ zeroes because the process never moves the rightmost digit leftwards. So the sequence $(k+1)S_n$ starts at $(k+1)S_{m^\prime}\ge (k+1)S_m>kS_m$ and thus, it appears later.
\end{proof}

The exact same argument proves that the sequence $A+B$ appears after sequences $A$ and $B$.

By the way, the sequence $a(n)$, where $a(n)$ is the row number in the Wythoff array that is $n$ times the Fibonacci sequence is A269725. The similar sequence for Lucas numbers is A269726.

\subsection{Row numbers for multiples}

\begin{theorem}
When all the numbers in an extraTrib are divisible by $n$, the row number is 1 modulo $n$.
\end{theorem}

\begin{proof}
Consider an extraTrib $S$ with all terms divisible by $n$. After dividing by $n$, we get another extraTrib that is a row in the Trithoff array. Suppose its element in the first column is $p$. Then our sequence $S$ contains an element $np$. Consider the Tribonacci representation $v = (p)_T$ and an improper word $w$, where we replace every digit one in the word $v$ with $n$. Suppose an improper word $w$ requires adding exactly $z$ trailing zeroes to canonize. Thus, the row for sequence $S$ starts with $[|w0^z|]_T$. The canonization procedure depends on the rule for the Tribonacci-like sequences but not on the sequences themselves. Thus, the canonization steps are identical for bases $T$ and $U$. By Lemma~\ref{lemma:rownumber}, the row number is $1 + [|w0^z|]_U = 1 + n[v0^z]_U$. This is 1 plus a multiple of $n$.
\end{proof}

For example, consider the Tribonacci sequence and its multiples. Suppose that some range of values of $n$ needs the same number of zeros $z$ to get canonized. In other words, for this range, the canonization of $n0^z$ ends in 1. That means the first column of the row that is the $n$th multiple of the Tribonacci sequence equals $nT_{n+2+z}$. Thus, for this range the elements in the first column form an arithmetic progression with difference $T_{n+2+z}$, and row numbers form an arithmetic progression with difference $T_{n+2+z} - T_{n+1+z}$.

\subsection{How to find the extraTribs in the garden}

Given an extraTrib, how can we locate it in the array? We can do this by computing the outs of each term. From Fact 5T, when $n$ is the wall, $\out(n)-1$ is the term after $n$, and it is the last term that does so. Moreover, suppose we have three consecutive terms of an extraTrib that are $m$, $\out(m)$, and $\out^2(m)$. Then $n+\out(n)+\out^2(n)=\out^3(n)$, it follows that the next terms is $\out^3(n)$ and so on.

Thus to find the wall term in an extraTrib, it is enough to locate a term $n$, such that the next term is $m = \out(n)-1$, and the next two terms are $\out(m)$ and $\out^2(m)$.

\subsection{Extending to the left}

Let us extend the Fibonacci sequence to the left:
\[\ldots,\ -8,\ 5,\ ,-3,\ 2,\ -1,\ 1,\ 0,\ 1,\ ,1\ 2.\]
We see that the signs on the left alternate. 

Note that any extraFib series extended to the left has a similar pattern. The signs on the left alternate, and the absolute values moving to the left form an extraFib, see \cite{ConwayRyba2016}.

Going backwards through the Tribonacci sequence gives 1, 0, 0, 1, $-1$, 0, 2, $-3$, 1, 4, $-8$, 5, 7, $-20$, 18, 9, $-47$, 56, 0, $-103$, 159, $-56$, $-206$, 421, $-271$, etc. We see that the signs do not form a nice pattern, and the absolute values do not form an extraTrib.

\begin{lemma}
\label{lemma:negativeextraTribs}
When extending an extraTrib to the left, we have to reach a negative number. After that, no three consecutive numbers to the left of it cannot have the same sign.
\end{lemma}

\begin{proof}
First, we prove that when extending an extraTrib sequence to the left, we always reach a negative number. Assume this is false, and there exists such an extraTrib whose elements are all non-negative. It follows that the sequence is non-decreasing as for any $k$, we have $T_k = T_{k-1} + T_{k-2} + T_{k-3} \geq T_{k-1}$. It follows that the sequence does not contain zeros and is actually monotonically increasing. There is a finite number of non-negative integers less than any given non-negative integer, but there are no bounds on the index moving to the left, so as we move to the left in an extraTrib sequence, we must at some point encounter a negative element.

Suppose $T_k < 0$. Let $T_n$, $T_{n-1}$, and $T_{n-2}$ for $n < k$ be all positive. This means that $T_{n+1} = T_n+T_{n-1}+T_{n-2} > 0$. By continuing, we see that for all $j > n$, we have $T_j > 0$, contradicting that $T_k < 0$. 

Let $T_n$, $T_{n-1}$, and $T_{n-2}$ for $n < k$ be all negative. This means that $T_{n+1} = T_n+T_{n-1}+T_{n-2} < 0$. By continuing, we see that for all $j > n$, we have $T_j < 0$, contradicting the fact that this is an extraTrib.
\end{proof}

\subsection{Reversal}

Similar to Conway and Ryba \cite{ConwayRyba2016}, we can define the reversal of an extraTrib series, the series where we change the index $n$ to $-n$ and replace the numbers with their absolute values. The following proposition is a negation of Fact 10.

\begin{theorem}
The reversal of the extraTrib is not an extraTrib.
\end{theorem}

\begin{proof}
We use the fact from Lemma~\ref{lemma:negativeextraTribs} that no three consecutive terms of an extraTrib can be all positive or all negative once the term index is below some constant $i_0$.

Moreover, for the reversal of the extraTrib series to be an extraTrib, we can assume that the terms to the left of some index constant $i_0$ increase in absolute value. Let the absolute values of some four consecutive terms $A$, $B$, $C$, $D$ to the left of $i_0$ be $a$, $b$, $c$, and $d$. By our assumptions, $$a > b > c > d.$$

Now consider the signs of $A$, $B$, $C$, and $D$. Without loss of generality, we can assume that the sign for $A$ is positive. As no three consecutive terms have the same sign, there are 5 cases for the distribution of signs: $+ + -+$, $+ + --$, $+ - ++$, $+ - +-$, and $+ - -+$. We have that $A+B+C=D$. Thus, we get the following equations: $a + b - c = d$, 
$a +  b - c = - d$, 
$a - b + c = d$, 
$a - b + c = - d$, and $a  - b -  c = d$. Or equivalently, 
$a + b = c + d$, $a + b + d = c$, 
$a + c = b + d$, $ a + c + d = b$, and $a = b + c + d$.

Given that $a > b > c > d$, we can exclude the first four cases. We are left with the case of four consecutive numbers $a$, $-b$, $-c$, and $d$. Consider the number to the left of $a$. On one hand, it has to equal $-c - (-b) - a = b - c - a$. On the other hand, its absolute value has to be $a+b+c$. We get a contradiction. 
\end{proof}

\section{Fib/Trib binary/ternary numbers}\label{sec:tribinarrynumbers}

The sequence of \textit{Fibbinary} numbers (A003714) is defined as numbers whose binary representation contains no two adjacent ones. In other words, the Fibbinary numbers can be formed by writing the Zeckendorf representations of natural numbers and then evaluating the result in binary:
\[0,\ 1,\ 2,\ 4,\ 5,\ 8,\ 9,\ 10,\ 16,\ 17,\ 18,\ 20,\ 21,\ 32,\ 33,\ 34,\ 36,\ 37,\ 40,\ \ldots. \]

Analogously, we define the \textit{Tribbinary} numbers as those numbers whose binary representation has no three consecutive ones. The sequence of Tribbinary numbers can be constructed by writing out the Tribonacci representations of non-negative integers and then evaluating the result in binary. This is sequence A003726: 
\[0,\ 1,\ 2,\ 3,\ 4,\ 5,\ 6,\ 8,\ 9,\ 10,\ 11,\ 12,\ 13,\ 16,\ 17,\ \ldots\]

Now we would like to introduce two more sequences related to base 3, rather than base~2. 

We define \textit{Fibternary} numbers as numbers whose ternary representations consist only of zeros and ones and do hot have two consecutive ones. The sequence of Fibternary numbers can be constructed by writing out the Zeckendorf representations of non-negative integers and then evaluating the result in ternary. This is sequence A060140:
\[0,\ 1,\ 3,\ 9,\ 10,\ 27,\ 28,\ 30,\ 81,\ 82,\ 84,\ 90,\ 91,\ 243,\ 244,\ \ldots.\]
These are Fibbinary numbers written in base 2, then evaluated in base 3.

We define \textit{Tribternary} numbers as numbers whose ternary representations consist only of zeros and ones and do hot have three consecutive ones. The sequence of Tribternary numbers can be constructed by writing out the Tribonacci representations of non-negative integers and then evaluating the result in ternary. This is now sequence A356823:
\[0,\ 1,\ 3,\ 4,\ 9,\ 10,\ 12,\ 27,\ 28,\ 30,\ 31,\ 36,\ 37,\ 81,\ 82,\ 84,\ 85,\ 90,\ 91,\ \ldots.\]
%0, 1, 3, 4, 9, 10, 12, 27, 28, 30, 31, 36, 37, 81, 82, 84, 85, 90, 91, 93, 108, 109, 111, 112, 243, 244, 246, 247, 252, 253, 255, 270, 271, 273, 274, 279, 280, 324, 325, 327, 328, 333, 334, 336, 729, 730, 732, 733, 738, 739, 741, 756, 757, 759, 760, 765, 766, 810, 811, 813, 814, 819, 820, 822, 837, 838, 840, 841, 972, 973, 975, 976, 981, 982, 984, 999, 1000
%Select[Range[0, 1000], 
% SequenceCount[IntegerDigits[#, 3], {1, 1, 1}] == 0 && 
%   SequenceCount[IntegerDigits[#, 3], {2}] == 0 &]
These are Tribbinary numbers written in base 2, then evaluated in base 3.

A lot is known about Fibbinary numbers and can be easily generalized to the other three sequences.

\textbf{Powers of 2 and 3.} The number of Fibbinary numbers less than any power of two is a Fibonacci number. It is easy to prove that the number of Tribbinary numbers less than any power of two is a Tribonacci number. Similarly, the number of Fibternary(Tribternary) numbers less than any power of three is a Fibonacci(Tribonacci) number.

\textbf{Recursive generation.}
We can generate all the four sequences we discuss here recursively. Start by adding 0 to the sequence. Then, if $x$ is a number in the sequence, add the following numbers to the sequence (ignoring repeated zeros):
\begin{itemize}
\item $2x$ and $4x+1$, for Fibbinary;
\item $2x$, $4x+1$, and $8x+3$, for Tribbinary;
\item $3x$ and $9x+1$, for Fibternary;
\item $3x$, $9x+1$, and $27x+4$, for Tribternary;
\end{itemize}

\textbf{Fibonacci(Tribonacci) word.} The Fibbinary numbers have the property that the $n$th Fibbinary number is even if the $n$th term of the Fibonacci word is $a$. Respectively, the $n$th Fibbinary number is odd (of the form $4x+1$) if the $n$th term of the Fibonacci word is $b$. 

Similarly, the $n$th Fibternary number is of the form $3x$ (correspondingly $9x+1$) if $n$th term of the Fibonacci word is $a$ (correspondingly $b$) (see comment in the OEIS for A060140).

Similarly, the $n$th Tribbinary number is even if the $n$th term of the Tribonacci word is $a$. Respectively, the $n$th Tribbinary number is of the form $4x+1$ if the $n$th term of the Tribonacci word is $b$, and the $n$th Tribbinary number is of the form $8x+3$ if the $n$th term of the Tribonacci word is $c$. This follows from Theorem~\ref{thm:abcCorrespondsTo0-01-11}, see \cite{DucheneRigo2008}.

Similarly, the $n$th Tribternary number is divisible by 3 if the $n$th term of the Tribonacci word is $a$. Respectively, the $n$th Tribbinary number is of the form $9x+1$ if $n$th term of the Tribonacci word is $b$, and the $n$th Tribbinary number is of the form $27x+4$ if $n$th term of the Tribonacci word is $c$.

\textbf{Sums.} It is known and can be easily checked, that every non-negative integer can be written as the sum of two Fibbinary numbers. As Fibbinary numbers are a subset of Tribbinary numbers, we get that every non-negative integer can be written as the sum of two Tribbinary numbers. Here is the analog for Fibternary and Tribternary numbers.

\begin{proposition}
Every non-negative integer can be written as a sum of four Fibternary numbers or as a sum of three Tribternary numbers.
\end{proposition}

\begin{proof}
We start with Fibternary numbers. Suppose base-3 representation of integer $n$ has $k$ digits.   Consider two $k$-digit Fibternary numbers of the form $101010\ldots$ in base 3 and another two $k$-digit numbers of the form $010101\ldots$. We can add these four numbers to get the $k$-digit number $N$ written as $222222\ldots$ in base 3. We can get to our number $n$ by subtracting one or two in some digit placements. We can distribute these subtractions between our four numbers by replacing some ones with zeros in them. The four numbers will remain Fibternary and will sum up to $n$.

We continue with Tribternary numbers. Suppose base-3 representation of integer $n$ has $k$ digits. Consider three special numbers in base 3, all of them consisting of zeros and ones. The first number has zeros in digit places divisible by 3, the second number in digit places that have remainder 1 when divided by three, and the third number in digit places that have remainder 2 when divided by 3. These three numbers sum up to a number with $k$ digits in base three, all equal to 2. Suppose our number $n$ in base 3 has the digit 1 in some place. Then we can remove a 1 from one of the two special numbers that have a 1 in the same place. Suppose our number $(n)_3$ has the digit 0 in some place. Then we can remove a 1 from both of the two special numbers with a 1 in the same place. When all digits are adjusted, we will have three numbers that sum to $n$, and all of them have every third digit as zero. Thus, all of them are Tribonacci representations of some numbers.
\end{proof}

\textbf{Multiples.} Every number has a Fibbinary multiple. The proof is available in the sequence A300867 entry in the OEIS \cite{OEIS}. Our generalization is in the next proposition.

\begin{proposition}
Every number has a Fibternary multiple.
\end{proposition}

\begin{proof}
Let $a(k)=\frac{9^k-1}{8}$, and then for any $n$, the pigeonhole principle implies there are $i\neq j$ such that $a(i)\equiv a(j)\mod n$, making $a(i)-a(j)$ a multiple of $n$. 

In addition,
$$9^k-1={\overbrace{888\ldots8}^{k\text{ 8's}}}_9,$$ so 
$$a(k)={\overbrace{111\ldots1}^{k\text{ 1's}}}_9.$$
Then $$a(i)-a(j)=\overbrace{111\ldots1}^{i-j\text{ 1's}}{\overbrace{000\ldots0}^{j\text{ 0's}}}_9=\overbrace{010101\ldots01}^{i-j\text{ 01's}}{\overbrace{000\ldots0}^{2j\text{ 0's}}}_3.$$

This means that $a(i)-a(j)$ is actually fibternary.
\end{proof}

As every Fibbinary number is also Tribbinary and every Fibternary number is also Tribternary, we have the following corollary.

\begin{corollary}
Every number has a Tribbinary and a Tribternary multiple.
\end{corollary}

\section{Acknowledgments}

We are grateful to PRIMES STEP program for giving us the opportunity to conduct this research.

\end{document}